\documentclass[12pt]{article}
\usepackage{setspace}
\usepackage{geometry}
\usepackage{color}
\usepackage{amsfonts}
\usepackage{amssymb}
\usepackage{amsmath}

\usepackage[mathscr]{eucal}
\usepackage{amsfonts}
\usepackage{amsmath,amsxtra,latexsym,amsthm, amssymb, amscd}
\usepackage[colorlinks,citecolor=blue,filecolor=black,linkcolor=blue,urlcolor=black]{hyperref}
\usepackage{pgfplots} 
\usepackage{url}
\usepackage[round]{natbib}
\usepackage{fancyhdr}
\usepackage{tikz}

\usepackage[round]{natbib}
\usepackage{fancyhdr}
\usepackage{makeidx}

\makeindex

\usepackage{graphicx}
\usepackage{hyperref}
\usepackage{subcaption}
\usepackage{caption}
\usepackage{epstopdf}

\usepackage{soul}

\newtheorem{example}{Example}

\newtheorem{lemma}{Lemma}

\newtheorem{proposition}{Proposition}
\newtheorem{remark}{Remark}

\newtheorem{assum}{Assumption}

\newcommand{\rr}{\mathbb{R}}

\begin{document}

\title{{\Large\bf A note on first-order and transversality conditions in infinite-horizon continuous-time optimal control models}}
\author{Stefano BOSI\thanks{
CEPS, Universit\'e Paris-Saclay. Email: stefano.bosi@univ-evry.fr} \and David DESMARCHELIER\thanks{BETA, University of Lorraine. Email: david.desmarchelier@univ-lorraine.fr} \and Ngoc-Sang PHAM\thanks{EM Normandie Business School, M\'etis Lab. Email: npham@em-normandie.fr} }
\date{\today}

\maketitle
\begin{abstract}

We provide simple and easily verifiable necessary and sufficient conditions for infinite-horizon continuous-time optimal control problems. Under standard concavity and integrability assumptions, we show that optimal paths are fully characterized by first-order conditions together with a transversality condition. Importantly, the transversality condition is derived as a consequence of the model’s structure rather than imposed a priori. Our results apply directly to standard economic models, including optimal growth and consumption–saving problems.
\\
\newline
{\it Keywords:} optimal control, infinite-horizon, continuous-time, maximum principle, transversality condition, optimal growth.\\
\textit{JEL Classifications}: C61, D15.

\end{abstract}


\section{Introduction}
Continuous-time optimal control models play a central role in economics, particularly in growth theory and macroeconomic dynamics. While the maximum principle provides a standard framework to characterize optimal paths, its application in infinite-horizon settings often raises practical difficulties.

Classic references such as \cite{clarke2013functional} and \cite{vinter2010optimal} provide rigorous treatments, but focus on finite-horizon settings. These textbooks are intended for mathematicians and, by the way, require  a solid background in mathematics, including functional analysis and non-smooth analysis, which may limit their accessibility to economists.

 \citet[Chapter 7]{Acemoglu2009introduction}   also introduces the optimal control theory. He presents Theorem 7.13 to show the maximum principle for discounted infinite-horizon problems. However, as  \cite{Acemoglu2009introduction} himself acknowledges in Section 7.9, verifying the assumptions in Theorem 7.13 is far from straightforward.

The textbook of \cite{SeierstadSydsaeter1987optimal} provides  an excellent introduction to the theory of optimal control and numerous applications in economic models. However, this book pays less attention\footnote{See discussions in Chapters 3 and 6 in \cite{SeierstadSydsaeter1987optimal}.} to the transversality conditions (TVCs for short) and significant progresses on TVCs have been done since its publication.

Within standard economic frameworks, the necessity of transversality conditions can be established using methods developed by \cite{EkelandScheinkman1986}, \cite{Michel1982}, \cite{Kamihigashi2001}. However, as shown by \cite{Halkin1974} in Section 5, transversality conditions may not be satisfied in general. 
 This suggests that transversality conditions should be derived from the structure of the model rather than imposed a priori. 

This note provides simple and easily verifiable necessary and sufficient conditions for infinite-horizon continuous-time optimal control problems. Under standard concavity and integrability assumptions, we show that optimal paths can be fully characterized by first-order conditions together with a transversality condition that emerges endogenously. 

Our results are directly applicable to commonly used economic models, including optimal growth and consumption–saving problems. By offering a tractable and unified set of conditions, this note aims to facilitate the use of continuous-time optimal control techniques in applied economic analysis.

\section{Problem statement and results}
We consider the following problem (P):
\begin{subequations}
\begin{align}
\text{Problem (P): } \quad \max_{c(\cdot),x(\cdot)}\int_{t_0}^{\infty}e^{-\theta t}u\left(c(t),x(t)\right)  dt,\\
\label{constraint1}c(t)+\dot{x}(t)= f(x(t),t) \text{ a.e.},\\
\label{constraint2}(x(t),\dot{x}(t))\in  \mathcal{X}_t\subseteq \rr^2,
\end{align}
\end{subequations}where $\theta>0$,   $c:\mathcal{T}\to \mathbb{R}$, $x:\mathcal{T}\to \mathbb{R}$ with $\mathcal{T}\equiv [t_0,\infty)$, and a.e. means that \eqref{constraint1} holds for almost every $t \in  \mathcal{T}$.\footnote{See Footnote 4 in \cite{Halkin1974}.} The initial value $x(t_0)$ is given. 

We introduce fundamental notions. A pair $(c(\cdot), x(\cdot))$ of functions is called admissible if it satisfies \eqref{constraint1}, \eqref{constraint2} and the function $x(\cdot)$ is piecewise differentiable\footnote{See Footnote 2 in \cite{Halkin1974}.}  (this ensures that the integral $\int_{t_0}^{T}e^{-\theta t}u\left(c(t),x(t)\right)dt$ is well defined for any $T>0$).\footnote{In the classical literature of optimal control \citep{vinter2010optimal,clarke2013functional}, it requires the absolute continuity.} 


An admissible pair $(c(\cdot), x(\cdot))$ is a solution to the problem (P) if $\int_{0}^{\infty}e^{-\theta t}u\left(c(t),x(t)\right)  dt\in (-\infty,\infty)$ and for any admissible pair $(\tilde{c}(\cdot), \tilde{x}(\cdot))$, we have 
\begin{align}\label{limsup_condition}
\int_{t_0}^{\infty}e^{-\theta t}u\left(c(t),x(t)\right)  dt\geq \limsup_{T\to\infty}\int_{t_0}^{T}e^{-\theta t}u\left(\tilde{c}(t),\tilde{x}(t)\right)dt,
\end{align}
or, equivalently, $\liminf_{T\to\infty}\Big(\int_{t_0}^{T}e^{-\theta t}u\left(c(t),x(t)\right)  dt- \int_{t_0}^{T}e^{-\theta t}u\left(\tilde{c}(t),\tilde{x}(t)\right)dt\Big)\geq 0$.\footnote{This definition of optimality corresponds to the catching-up criterion optimality in  \citet[Chapter 3, Section 7]{SeierstadSydsaeter1987optimal}. \cite{Kamihigashi2001} considers the criterion $\limsup_{T\to\infty}\Big(\int_{0}^{T}e^{-\theta t}u\left(c(t),x(t)\right)  dt- \int_{0}^{T}e^{-\theta t}u\left(\tilde{c}(t),\tilde{x}(t)\right)dt\Big)\geq 0$, which is weaker.}

Following \cite{Kamihigashi2001}, we introduce useful notions.
A solution $(c^*(\cdot), x^*(\cdot))$ is said to be interior if for all $t$, there exists $\epsilon>0$ such that $\forall s\in (t-\epsilon,t+\epsilon)$, the ball $B((x^*(s),\dot{x}^*(s)), \epsilon)\subset \mathcal{X}(s)$, where $B(a, \epsilon)\equiv \{a'\in \rr^2: \|a'-a\|<\epsilon\}$  $\forall a\in \rr^2$ and $\|\cdot\|$ is the Euclidean distance.

A solution $(c^*(\cdot), x^*(\cdot))$ is said to be regular if it satisfies regularity conditions: (R1) for all $t$, there exists an open set $J(t)\subset \mathcal{X}_t$ such that $(x^*(t),\dot{x}^*(t))\in J(t)$ and the function $v_t: (y,z)\mapsto u(f(y,t)-z,y)$ is in $C^1$ on $J(t)$, (R2) the functions $p_t:t\mapsto \frac{\partial v_t}{\partial y}(x^*(t),\dot{x}^*(t))$ and $q_t: t\mapsto \frac{\partial v_t}{\partial z}(x^*(t),\dot{x}^*(t))$ are piecewise continuous and right continuous, (R3) for all $t$, there exist $\epsilon,B>0$ such that $(y,z)\in J(s)$, $|\frac{\partial v_s}{\partial y}(y,z)|\leq B$ and $|\frac{\partial v_s}{\partial z}(y,z)|\leq B$ $\forall s\in (t-\epsilon,t+\epsilon), \forall (y,z)\in B((x^*(s),\dot{x}^*(s)), \epsilon)$.\footnote{Conditions (R1), (R2), (R3) respectively correspond to Assumptions (A4), (A5), (A6) in \cite{Kamihigashi2001}. 
}


Let us mention some particular cases of the problem (P) widely used in economics.
\begin{enumerate}
\item When $\mathcal{X}_t=\{(x,y)\in \rr^2:f(x,t)-y\in\mathcal{U}\}$, where $\mathcal{U}$ is a fixed subset of $\rr$, constraint \eqref{constraint2} becomes $c(t)\in \mathcal{U}\subseteq \rr$.\footnote{This corresponds to constraint (170) in \citet[p.~230]{SeierstadSydsaeter1987optimal}.}

\item Let $\mathcal{X}_t=\{(x,y)\in \rr^2: x\geq 0, f(x,t)-y\geq 0\}$. Assume that the function $u\left(\cdot,x\right)$ does not depend on $x$. We can simply write $u\left(c,x\right)=u(c)$, where $u: \rr_+\to \rr_+$ is the utility function. By the way, we recover an optimal growth model with the constraints $c(t)+\dot{x}(t)=f(x(t),t), x(t)\geq 0$, where $x(t)$ represents the capital and $c(t)$ the consumption. 

In this case, a solution $(c^*(\cdot), x^*(\cdot))$ is interior and regular if $c^*(t), x^*(t)>0$ $\forall t$, and both $u$ and $f$ are continuously differentiable.
\end{enumerate}

Let us define the functions  $L: \rr\times \rr\times \rr_+\to \rr$, $G: \rr\times \rr\times \rr_+\to \rr$ and $H: \rr\times \rr\times \rr_+\times \rr^2\to \rr$ as follows.
\begin{subequations}
\label{Assumption_LGH}
\begin{align}
L(x,c,t)&\equiv e^{-\theta t}u\left(  c,x\right), \quad G(x,c,t)\equiv f(x,t)-c\\
H(x,c,t,\lambda_0,\lambda)&\equiv \lambda_0 L(x,c,t)+\lambda  G(x,c,t)=\lambda_0 e^{-\theta t}u\left(c,x\right)+\lambda  \big(f(x,t)-c\big).
\end{align}
\end{subequations}

\begin{assum}\label{assumption_basic}
The function $f:\rr\times \rr_+\to \rr$ is continuously differentiable.  The function $u:\rr\times \rr\to \rr$ is continuously differentiable, non-decreasing in each component and $\frac{\partial u}{\partial c}(c,x)>0$ $\forall (c,x)\in \rr_{++}^2$.
\end{assum}

\begin{assum}\label{concave_assumption} For each $(t,\lambda_0,\lambda)$, the function $H(\cdot, \cdot, t,\lambda_0,\lambda)$ is concave.


\end{assum}

\subsection{Necessary conditions}
We introduce some assumptions playing an important role in proving the transversality conditions.
\begin{assum}[Concavity]
\label{assum_Ut}The function $f(\cdot,t)$ is concave and $f(0,t)\geq 0$ $\forall t$. Moreover, there exists $\lambda_1\in [0,1)$ such that 
$\lambda \xi\in \mathcal{X}_t$ $\forall \xi\in \mathcal{X}_t,\forall \lambda\in [\lambda_1,1).$ 
\end{assum}


This is a variant of Assumption 3.3 in \citet{Kamihigashi2001}.  Assumption \ref{assum_Ut} holds in standard economic models, under mild concavity conditions.
\begin{assum}\label{tvc_assumption_general}
    There exist $\theta^*,\theta^*_0\in \rr$ and $\underline{\lambda}\in (0,1)$ such that $\frac{u(c,x)-u(\lambda c,\lambda x)}{1-\lambda}\leq \theta^* u(c,x) +\theta^*_0$ $\forall \lambda\in (\underline{\lambda}, 1), \forall (c,x)\in \{(c',x'): u(c',x')>-\infty\}$.
\end{assum}

This is  a variant of Assumption (ii) in Proposition 5.1 in  \cite{EkelandScheinkman1986} and  Assumption 3(2) in \cite{Pham2025_relationship}. As pointed out by \cite{Pham2025_relationship}, this assumption is satisfied in standard setups, for instance, when $u(c,x)=c^{1-\sigma}/(1-\sigma)$, where $0<\sigma\not=1$ or $u(c)=\ln(c)$. It also holds when $u(0)>-\infty$.

\begin{proposition}\label{prop_necessary}
Consider the problem (P). Let Assumptions \ref{assumption_basic} and \ref{concave_assumption} be satisfied. Consider a solution $(c^*(\cdot), x^*(\cdot))$ which is  interior and regular. 

(1) Then there exists a piecewise  differentiable function $\lambda:  \mathcal{T}\to \rr$ such that 
\begin{subequations}\label{foc_general_prop1}
\begin{align}\label{foc_general_1}
&e^{-\theta t}\frac{\partial u(c^*(t),x^*(t))}{\partial c}=\lambda(t),\\
\label{foc_general_2}&\dot{\lambda}(t) = -e^{-\theta t}\frac{\partial u(c^*(t),x^*(t))}{\partial x}-\lambda(t) \frac{\partial f^*}{\partial x}(x^*(t),t).
\end{align}\end{subequations}


(2) If Assumptions \ref{assum_Ut} and \ref{tvc_assumption_general} hold,\footnote{Actually, in our proof, we only need condition in Assumption \ref{tvc_assumption_general} for any $c\in \{c^*_t:t\geq 0\}$, but not necessarily for any $c\in \{z: u(z)>-\infty\}$.} and $\int_{0}^{\infty}e^{-\theta t}|u( c^*(t),x^*(t))| dt<\infty$, we have the so-called transversality condition \begin{align}
\label{tvc_condition_general}
    \limsup_{t\to\infty}{\lambda}(t) x^*(t)&\leq 0.
\end{align}
By consequence, if there exists $t_1$ such that $x^*(t)\in \rr_+$ $\forall t\geq t_1$, then $\lim_{t\to\infty}{\lambda}(t)x^*(t)=0.$
\end{proposition}
\begin{proof}[Proof of Proposition \ref{prop_necessary}]
Applying Theorem 3.1 in \cite{Kamihigashi2001}, we have  first-order conditions \eqref{foc_general_prop1}. 
We now prove the TVC \eqref{tvc_condition_general}. We make use of Theorem 3.3(ii) in \cite{Kamihigashi2001}. We do so by verifying two conditions. 
First, we check Assumption 3.2 in \cite{Kamihigashi2001}. Note that the set $X_t$ in \cite{Kamihigashi2001} is defined in our model as follows:
$$X_t\equiv  \{(y,z)\in \rr\times \rr: 
f(y,t)-z\in \rr_+, (y,z)\in \mathcal{X}_t\}.$$

Let $\tilde{\lambda}\in (0,1)$ and take $\lambda \in [\tilde{\lambda},1)$. We claim that  $(\lambda x^*(t), \lambda\dot{x}^*(t))\in X_t$. By Assumption \ref{assum_Ut}, $(\lambda x^*(t), \lambda\dot{x}^*(t))\in \mathcal{X}_t$. By the definition of $(c^*(\cdot), x^*(\cdot))$, we have 
$$f(x^*(t),t)-\dot{x}^*(t)\in \rr_+.$$
Multiplying by $\lambda \in [\tilde{\lambda},1)$, we get $\lambda f(x^*(t),t)-\lambda \dot{x}^*(t)\in \rr_+.$ Since the function $f(\cdot,t)$ is concave and $f(0,t)\geq 0$ $\forall t$, we have $\lambda f(x^*(t),t)\leq  f(\lambda x^*(t),t)$. Hence, $f(\lambda x^*(t),t)-\lambda \dot{x}^*(t)\in \rr_+$, and hence $(\lambda x^*(t), \lambda\dot{x}^*(t))\in X_t$. To sum up,  Assumption 3.2 in \cite{Kamihigashi2001} holds.

Second, we prove the condition in the statement of Theorem 3.3(ii) in \cite{Kamihigashi2001}.  Take  
$\underline{\lambda}$ as in Assumption \ref{tvc_assumption_general}. Then,  take $\tilde{\lambda}=\underline{\lambda}$ and $\lambda \in [\tilde{\lambda},1)$. 

For $t\geq 0, y\geq 0,z\in \rr$, we define
\begin{align}
c(y,z,t) &\equiv f(y,t)-z, & v(y,z,t)&\equiv e^{-\theta t}u(c(y,z,t),y).
\end{align}
We have
\begin{subequations}
\begin{align}
W(\lambda,t)&\equiv \frac{v(x^*(t), \dot{x}^*(t),t)-v(\lambda x^*(t), \lambda \dot{x}^*(t),t)}{1-\lambda} \\
&=e^{-\theta t}\frac{u\big(c(x^*(t), \dot{x}^*(t),t),x^*(t)\big)-u\big(c(\lambda x^*(t), \lambda \dot{x}^*(t),t),\lambda x^*(t)\big)}{1-\lambda}.
\end{align}\end{subequations}
Since the function $f(\cdot,t)$ is concave, we have $f(\lambda x,t)\geq \lambda f(x,t)$ $\forall \lambda\in (0,1), \forall x\in \rr,  \forall t\in \mathcal{T}$. By consequence,
\begin{subequations}
\begin{align}
c(\lambda x^*(t), \lambda \dot{x}^*(t),t)&=f(\lambda x^*(t),t)- \lambda \dot{x}^*(t)\\
&\geq \lambda \Big(f(x^*(t),t)- \dot{x}^*(t)\Big)=\lambda c(x^*(t), \dot{x}^*(t))  \text{ a.e.}
\end{align}
\end{subequations}


By combining with the monotonicity of $u$, we have $u\big(c(\lambda x^*(t), \lambda \dot{x}^*(t),t),\lambda x^*(t))\big)\geq u\big(\lambda c(x^*(t), \dot{x}^*(t),t),\lambda x^*(t)\big)$.
Thus, 
\begin{subequations}
\begin{align}
W(\lambda,t)&\leq e^{-\theta t}\frac{u\big(c(x^*(t), \dot{x}^*(t),t),x^*(t)\big)-u\big(\lambda c( x^*(t), \dot{x}^*(t),t),\lambda x^*(t)\big)}{1-\lambda}\\
&\leq \theta^* u\big(c(x^*(t), \dot{x}^*(t),t),x^*(t)\big) +\theta^*_0\\
&=\theta^* u\big(c^*(t),x^*(t)\big) +\theta^*_0,
\end{align}\end{subequations}
where the second inequality follows Assumption \ref{tvc_assumption_general}.

By Assumption $\int_{t_0}^{\infty}e^{-\theta t}|u( c^*(t),x^*(t))| dt<\infty$, we conclude that the conditions in the statement of Theorem 3.3(ii) in \cite{Kamihigashi2001} is satisfied. By consequence, we have $\limsup_{t\to\infty}\lambda(t) x^*(t)\leq 0$.

By \eqref{foc_general_1} and Assumption \ref{assumption_basic}, we have $\lambda(t)\in \rr_+$. If $ x^*(t)\in \rr_+$ $\forall t$, then we have $\lambda(t)  x^*(t)\geq 0$ $\forall t$, which implies that  $\lim_{t\to\infty}\lambda(t)  x^*(t)=0$.

\end{proof}

In Proposition \ref{prop_necessary}, we focus on interior and regular solutions and we can write the maximum principle in a normal form.\footnote{See \cite{AseevKryazhimskii2007pontryagin}, \cite{AseevVeliov2014needle} and references therein for some normal form versions of the Pontryagin maximum principle.}  When state constraints take a general form $(x(t),\dot{x}(t))\in  \mathcal{X}_t\subseteq \rr^2$, it is not easy to write a maximum principle for a solution which is not necessarily interior (see \cite{frankowska2010optimal}, \cite{seierstad2015maximum}, \cite{Bascoetal2018necessary} among others and references therein).

However, in some economic settings (as in Proposition  \ref{interior_sufficient} below), we can provide sufficient conditions ensuring that a solution is both interior and regular. To establish this result, we impose additional assumptions.
\begin{assum}\label{interior_assumption}
\begin{enumerate}
\item \label{lowerbound_x} There exists $\underline{x}$ such that
\begin{enumerate}
\item $\lim_{c\to 0^+}\frac{\partial u(c,x)}{\partial c}=+\infty$ $\forall x>\underline{x}$ (Inada condition),
\item  $x(t_0)>\underline{x}$, 
\item \label{interior_assumption_DE} For all $\tau\geq t_0$, $x>\underline{x}$, 
any solution $x_0(\cdot)$ of the differential equation $\dot{x}(t)=f(x(t),t) $ $\forall t\geq \tau$ with the initial value $x(\tau)=x$ satisfies  $x_0(t)>\underline{x}$ $\forall t\geq \tau$.\footnote{This assumption obviously holds when $f(x,t)\geq 0$ $\forall x,t$, which is standard in economics.}
\end{enumerate}
\item\label{term2_bounded}  (Discounting condition) There exists $t_1>t_0$ such that 
$\int_{t_0}^\infty e^{-\theta t} |\frac{\partial u}{\partial x}(0, x_0(t)) y(t)| \, dt<\infty$,
where $x_0(\cdot)$ is a solution to the differential equation $\dot{x}(t)=f(x(t),t) $ $\forall t\geq t_0$ with the initial value $x(t_0)$,   while $y$ is a solution to 
\begin{align}\label{DEy_3}
\dot{y}(t)&=A(t)y(t)+b(t) \text{ } \forall t\geq t_0, \quad y(t_0)=0,
\end{align}
where $A(t)\equiv \frac{\partial f}{\partial x}(x_0(t),t)$ and $b$ is defined by $b(t)=-1$ if $t\in [t_0,t_1]$ and  $b(t)=0$ if $t\in (t_1,\infty)$.\footnote{We can compute $y(t)=-\int_{t_0}^t e^{\int_s^t A(\tau)\,d\tau}ds  \text{ if } t\in [t_0,t_1]$ and $y(t)= -\int_{t_0}^{t_1} e^{\int_s^t A(\tau)\,d\tau}ds  \text{ if } t\in (t_1,\infty).$
For $t > t_1$, the variation $b(t) = 0$. Hence, $y(t) = y(t_1) e^{\int_{t_1}^t f_x(x^*(s), s) \, ds}$ $\forall t > t_1.$
}

\end{enumerate}
\end{assum}


\begin{proposition}\label{interior_sufficient}
Consider the problem (P) for the case where \eqref{constraint2} is replaced by $c(t)\in \rr_+$. 
 Let Assumptions \ref{assumption_basic}, \ref{concave_assumption} and \ref{interior_assumption} be satisfied and $\rr_+\times (\underline{x},\infty)\subseteq dom(u)$, where $dom(u)$ denotes the domain of $u$.
Then any solution satisfies $c(t)>0$ almost everywhere and point 1 of Proposition \ref{prop_necessary} holds.
\end{proposition}
\begin{proof}See Appendix \ref{interior_sufficient_proof}.
\end{proof}
Let us illustrate this result by considering a standard economic setup where $u(c,x)=a(c)+v(x)$ with $a'(0)=+\infty$, $x(t_0)>0$ and  $f(x,t)=R(t)x$ with $R(t)>0$ the capital return. Indeed, in this case, we have $x_0(t)= x(t_0) e^{\int_{t_0}^t R(s) ds}>0$ $\forall t \geq  t_0$. So, Assumption \ref{interior_assumption}.\ref{lowerbound_x} holds for $\underline{x}=0$.

Assumption \ref{interior_assumption}.\ref{term2_bounded} obviously holds if $v(x)=0$ $\forall x$. Consider a case where $v(x)=x^{1-\sigma}/(1-\sigma)$, where $\sigma>0$, and assume that $R(t)=R>0$ $\forall t$. Then, we have $y(t)=y(t_1)e^{R(t-t_1)}$ $\forall t>t_1$ and $x_0(t)=x(t_0)e^{R(t-t_0)}$. So, Assumption \ref{interior_assumption}.\ref{term2_bounded}  holds if 
\begin{align*}
\int_{t_0}^\infty e^{-\theta t} |\frac{\partial u}{\partial x}(0, x_0(t)) y(t)| \, dt<\infty
\Leftrightarrow |y(t_1)||x(t_0)|^{-\sigma}\int_{t_0}^\infty e^{-\theta t} e^{-\sigma R(t-t_0)} e^{R(t-t_1)}dt<\infty.
\end{align*}
This is satisfied if $\theta+\sigma R>R$. This condition is related to the dominating discount conditions (Assumption (A4) in \cite{AseevVeliov2014needle} or (A7) in \cite{AseevKryazhimskii2007pontryagin}), which play an important role in establishing  normal
form versions of the  maximum principle.

\subsection{Sufficient conditions}


\begin{proposition}\label{prop_sufficient}

Let Assumption \ref{concave_assumption} be satisfied. Assume also that $\mathcal{X}_t=\rr_+\times \rr$ $\forall t$. An admissible pair $(c^*(\cdot), x^*(\cdot))$ is a solution if $\lambda_0=1$ and the following conditions hold.
\begin{enumerate}
\item $\int_{t_0}^{\infty}e^{-\theta t}u\left(  c^*(t),x^*(t)\right)  dt\in (-\infty,\infty)$.

\item There exists a continuously differentiable function $\lambda: \mathcal{T}\to \rr$ such that  
\begin{subequations}\label{foc_general}
\begin{align}\label{foc1_general}
\frac{\partial H^*}{\partial c}(x^*(t),c^*(t),t,\lambda_0,\lambda(t))&=0,\\
\label{foc1_general_Hx}\frac{\partial H^*}{\partial x}(x^*(t),c^*(t),t,\lambda_0,\lambda(t))&=-\dot{\lambda}(t),\\
\label{tvc_h}\lim_{t\to\infty}{\lambda}(t) x^*(t)&=0.
\end{align}
\end{subequations}

\end{enumerate}
\end{proposition}
\begin{proof}
Let $(\tilde{c}(\cdot), \tilde{x}(\cdot))$ be an admissible pair. By using the standard approach (see \citet[Chapter 2, Section 6]{SeierstadSydsaeter1987optimal} for instance) and the concavity of $H$ and \eqref{foc1_general}-\eqref{foc1_general_Hx}, we can prove that 
\begin{align}
\int_{t_0}^{t_1}\Big(e^{-\theta t}u\left(  c^*(t),x^*(t)\right)-e^{-\theta t}u\left(  \tilde{c}(t),\tilde{x}(t)\right)\Big)dt\geq & {\lambda}(t_1) \big[\tilde{x}(t_1))-x^*(t_1)\big]\geq -{\lambda}(t_1)  x^*(t_1)\label{tvc_remark}
\end{align}
where we use $\tilde{x}(t)\in \rr_+ $ $\forall t$ in the last inequality.

Let $t_1$ tend to infinity, we obtain that
\begin{align*}
\liminf_{t_1\to\infty}\int_{t_0}^{t_1}\Big(e^{-\theta t}u\left(  c^*(t),x^*(t)\right)-e^{-\theta t}u\left(  \tilde{c}(t),\tilde{x}(t)\right)\Big)dt\geq \liminf_{t_1\to\infty}\big(-{\lambda}(t_1)  x^*(t_1)\big)=0
\end{align*}
because of our assumption $\lim_{t\to\infty}{\lambda}(t)  x^*(t)=0$.

\end{proof}
\begin{remark}According to \eqref{tvc_remark}, the conclusion of  Proposition \ref{prop_sufficient} still holds if conditions $\mathcal{X}_t=\rr_+\times \rr$ $\forall t$ and \eqref{tvc_h}  in Proposition \ref{prop_sufficient} are replaced by $\liminf_{t\to\infty}{\lambda}(t_1) \big[\tilde{x}(t_1))-x^*(t_1)\big]\geq 0$ for any admissible pair $(\tilde{c}(\cdot), \tilde{x}(\cdot))$. This is basically a variant of the Arrow and Mangasarian sufficient conditions. 
\end{remark}
\begin{remark}
Proposition \ref{prop_sufficient} holds if we replace  (\ref{foc1_general}) by $H(x^*(t),c^*(t),t,\lambda_0,\lambda(t))\geq H(x^*(t),c,t,\lambda_0,\lambda(t))$ $\forall c\geq 0$. 
\end{remark}

\section{Applications}

We consider the following problem:
\begin{subequations}
    \label{application2}
\begin{align}
&\max_{c(\cdot),x(\cdot)}\int_{0}^{\infty}e^{-\theta t}u\left(c(t),x(t)\right)  dt,\\
\text{Constraints: }&c(t)+\dot{x}(t)=f(x(t),t)\equiv R(t)x(t)+\omega(t) \text{ a.e.},\\
&x(t)\geq 0,
\end{align}
\end{subequations}  
where $c(t),x(t)$ respectively represent the consumption and wealth of this agent. Here, we consider a variant of the model with wealth effects \citep{Kurz1968}.
Assume that the return and endowment functions $R,\omega: \mathbb{R}_+\to \mathbb{R}_+$ are differentiable.  

Assume that the function $u$ is strictly increasing, twice continuously differentiable, concave and satisfies Assumption \ref{tvc_assumption_general}. Consider a solution  $(c(\cdot),x(\cdot))$, which satisfies $c(t),x(t)>0$ $\forall t$. According to Proposition \ref{prop_necessary}, we have the FOCs
\begin{subequations}\label{foc_general_ex2}
\begin{align}\label{foc_general_1_ex2}
&e^{-\theta t}\frac{\partial u(c(t),x(t))}{\partial c}=\lambda(t),\\
\label{foc_general_2_ex2}&\dot{\lambda}(t) = -e^{-\theta t}\frac{\partial u(c(t),x(t))}{\partial x}-\lambda(t) \frac{\partial f}{\partial x}(x(t),t).
\end{align}\end{subequations}


Let us focus on a separable utility $u(c)+v(x)$, where both function $u,v$ are increasing, twice continuously differentiable, concave and satisfy the following assumption (which is actually Assumption \ref{tvc_assumption_general}).
\begin{assum}\label{tvc_assumption}
    There exist $\theta^*,\theta^*_0\in \rr$ and $\underline{\lambda}\in (0,1)$ such that $\frac{u(c)-u(\lambda c)}{1-\lambda}\leq \theta^* u(c) +\theta^*_0$ and $\frac{v(x)-v(\lambda x)}{1-\lambda}\leq \theta^* v(x) +\theta^*_0$  $\forall \lambda\in (\underline{\lambda}, 1), \forall c\in \{z_1: u(z_1)>-\infty\}, \forall x\in \{z_2: v(z_2)>-\infty\}$.
\end{assum}
Under Assumption \ref{tvc_assumption} and $\int_{0}^{\infty}e^{-\theta t}|u(c(t),x(t))| dt<\infty$, conditions in part 2 of Proposition \ref{prop_necessary} hold. By consequence, we have the following TVC
\begin{align}
    \label{tvc_condition_ex2}
  &  \lim_{t\to\infty}e^{-\theta t}u'(c(t))x(t)=0.
\end{align}
Note that \eqref{tvc_condition_ex2} holds if $\liminf_{t\to 0}c(t)>0$ and $ \lim_{t\to\infty}e^{-\theta t}x(t)=0$. 

Since $e^{-\theta t}u'(c(t))=\lambda(t)$, we can rewrite the system \eqref{foc_general_ex2} as follows.
\begin{align}\label{foc_special}
  &\frac{\dot{c}(t)}{c(t)}=\frac{(R(t)-\theta)u'(c(t))+v'(x(t))}{-c(t)u^{\prime \prime}(c(t))}.
\end{align}

\begin{example}\label{example1}
{\normalfont
Assume that $v(x)=0$ $\forall x$, $R(t)=R>0, \omega(t)=\omega>0$ $\forall t$ and consider the logarithmic utility ($u(c)=\ln(c)$ $\forall c$). According to \eqref{foc_special}, we have $\dot{c}(t)=(R-\theta)c(t)$ $\forall t$. This gives $c(t)=e^{(R-\theta)t}c(0) \text{ }\forall t$. To find $x(t)$, we use $\dot{x}(t)=Rx(t)+\omega-c(t)$, which means that $\dot{x}(t)-Rx(t)=\omega-e^{(R-\theta)t}c(0).$ This can be rewritten as $\frac{d}{dt}\big(x(t)e^{-Rt}\big)=e^{-Rt}\omega -c(0)e^{-\theta t}$. Then, we  find that
%
\begin{align}  
\label{find_xt}x(t)&=\frac{c(0)}{\theta}e^{(R-\theta)t}+\big(x(0)+\frac{\omega}{R}-\frac{c(0)}{\theta}\big)e^{Rt}-\frac{\omega}{R}.
\end{align}


It remains to find $c(0)$.
It is possible to check that $\int_{0}^{\infty}e^{-\theta t}|\ln(c(t))| dt<\infty$  and  $\int_{0}^{\infty}e^{-\theta t}\ln(c(t)) dt<\infty$. 
Then, by Proposition \ref{prop_necessary}'s part 2, we have the transversality condition $\lim_{t\to\infty}e^{-\theta t}x(t)/c(t)=0$, which allows us to find $c(0)=\theta (x(0)+\omega/R)$.  To sum up, the solution is given by $c(t)=e^{(R-\theta)t}\theta (x(0)+ \omega/R)\text{ }\forall t$.
}
\end{example}

\begin{remark}
Practically, we can prove that a pair $({c}(\cdot),{x}(\cdot))$  is the unique solution as follows. Step 1: we check, by using  Proposition \ref{prop_sufficient} that this is a solution to the problem.\footnote{Of course, we need Proposition \ref{prop_necessary} to identify candidates $({c}(\cdot),{x}(\cdot))$ (see Example \ref{example1}).} Step 2: we prove this is the unique solution by using the strict concavity of the objective function and the concavity of the admissible set. Indeed, if $(\tilde{c}(\cdot),\tilde{x}(\cdot))$ is another solution, we define $(c_{\gamma}(\cdot),x_{\gamma}(\cdot))$ by $c_{\gamma}=\gamma c+(1-\gamma)\tilde{c},x_{\gamma}=\gamma x+ (1-\gamma)\tilde{x}$, where $\gamma\in (0,1)$. This is, of course, admissible. Since $u$ is strictly concave, we have $$\int_{0}^{\infty}e^{-\theta t}u(c_{\gamma}(t),x_{\gamma}(t)) dt>\gamma \int_{0}^{\infty}e^{-\theta t}u(c(t),x(t)) dt+(1-\gamma)\int_{0}^{\infty}e^{-\theta t}u(\tilde{c}(t),\tilde{x}(t)) dt,$$
which is a contradiction. We conclude that it is the unique solution.

 \end{remark}

\begin{remark}In many economic applications, the analysis focuses on behavior in a neighborhood of steady states. Note that we have $\lim_{t\to\infty}e^{-\theta t}\frac{\partial u(c(t),x(t))}{ \partial c}x(t)=0$ if $\liminf_{t\to \infty}c(t)$, $\limsup_{t\to \infty}c(t)$, $\liminf_{t\to \infty}x(t)$, $\limsup_{t\to \infty}x(t)$ are in $(0,\infty)$. In such cases, first-order conditions together with a mild concavity assumption (as in Proposition \ref{prop_sufficient}) are sufficient to ensure the optimality of a trajectory. 
\end{remark}

\appendix
\section{Appendix: Proof of Proposition \ref{interior_sufficient}}
\label{interior_sufficient_proof}
{\bf Step 1 (maximum principle in a normal form)}. Applying  the maximum principle (see  \cite{Halkin1974} or Theorem 12 in  \citet[p.~234]{SeierstadSydsaeter1987optimal}), there exist a constant $\lambda_0\in \{0,1\}$ and a continuous and piecewise continuously differentiable function $\lambda(\cdot)$ such that  $(\lambda_0,\lambda(t))\not=(0,0),$
\begin{subequations}
\begin{align}
\label{optimal_c_general}H(x^*(t),c^*(t),t,\lambda_0,\lambda(t))&\geq H(x^*(t),c,t,\lambda_0,\lambda(t)) \text{ }\forall c\geq 0,\\
\label{Hh_derivative}\frac{\partial H^*}{\partial x}(x^*(t),c^*(t),t,\lambda_0,\lambda(t))&=-\dot{\lambda}(t) \text{ for almost every } t\in \mathcal{T}.
\end{align}
\end{subequations}
If $\lambda_0=0$, then $\lambda(t)\not=0$. Thus, \eqref{optimal_c_general} implies that 
\begin{align}
\lambda(t)\big(f(x^*(t),t)-c^*(t)\big)\geq \lambda(t)\big(f(x^*(t),t)-c\big) \text{ }\forall c\geq 0,
\end{align}
or, equivalently, $\lambda(t)\big(c-c^*(t)\big)\geq 0 \text{ }\forall c\geq 0.$ Take $c>c^*(t)$, we get $\lambda(t)>0$. By consequence, we have $c-c^*(t)\geq 0 \text{ }\forall c\geq 0,\forall t\geq t_0$. This implies that $c^*(t)=0$ $\forall t\geq t_0$, a contradiction (because of Lemma \ref{nozero_condition} below).  Therefore, we have $\lambda_0\not=0$. Hence, we get  $\lambda_0=1$. 
\begin{lemma}\label{nozero_condition}
There does not exist a solution satisfying $c^*(t)=0$ $\forall t\in \mathcal{T}$.
\end{lemma}
\begin{proof}
Suppose on the contrary that there is a solution with $c^*(t)=0$ $\forall t\in \mathcal{T}$. Then we have $\dot{x}^*(t)=f(x^*(t),t)$ $\forall t$. By Assumption \ref{interior_assumption}.\ref{interior_assumption_DE}, we have $x^*(t)>\underline{x}$ $\forall t$. 

 Take $t_1$ as in Assumption \ref{interior_assumption}.\ref{term2_bounded}. 
Given $\epsilon>0$, define $c_{\epsilon}(\cdot)$ by  $c_{\epsilon}(t)=\epsilon$ if $t\in (t_0,t_1]$, and $c_{\epsilon}(t)=0$ if $t\in (t_1,\infty)$.

We can choose  $\epsilon>0$ small enough, such that
the following differential equation 
\begin{subequations}
\begin{align}\label{DEx_1}
\epsilon+\dot{x}_{\epsilon}(t)&=f(x_{\epsilon}(t),t) \text{ }\forall t\in [t_0,t_1],\\
\label{DEx_2}\dot{x}_{\epsilon}(t)&=f(x_{\epsilon}(t),t)  \text{ }\forall t\in (t_1,\infty).
\end{align}
\end{subequations}
with the initial value $x_{\epsilon}(t_0)=x(t_0)$ has a solution ${x}_{\epsilon}(\cdot)$ satisfying ${x}_{\epsilon}(t)>\underline{x}$ $\forall t$.\footnote{We establish this claim in two steps. First, since $f$ is continuously differentiable (implying local Lipschitz continuity in $x$) and $x_0(t) > \underline{x}$ on the compact interval $[t_0, t_1]$, the Extreme Value Theorem guarantees that $m \equiv \min_{t \in [t_0, t_1]} (x_0(t) - \underline{x}) > 0$. By standard continuous dependence of ODE solutions on parameters (see \citet[Section 2.3]{Perko2001differential} among others), $x_\epsilon(t)$ converges uniformly to $x_0(t)$ on $[t_0, t_1]$ as $\epsilon \to 0^+$. Hence, we can choose $\epsilon > 0$ sufficiently small such that $\sup_{t \in [t_0, t_1]} |x_\epsilon(t) - x_0(t)| < m/2$, which strictly ensures $x_\epsilon(t) > \underline{x}$ for all $t \in [t_0, t_1]$. Second, using Assumption~\ref{interior_assumption}.\ref{interior_assumption_DE} alongside the interior condition $x_{\epsilon}(t_1) > \underline{x}$, we obtain the global existence of $x_{\epsilon}$ on $[t_1,\infty)$ satisfying \eqref{DEx_2} with initial value $x_{\epsilon}(t_1)$.} So, $(c_{\epsilon}(t),x_{\epsilon}(t))$ is in the domain of $u$ for any $t$.

Denote $A(t)\equiv \frac{\partial f}{\partial x}(x^*(t),t)$ and define ${y}(t)\equiv \lim_{\epsilon\to 0^+}\frac{{x}_{\epsilon}(t)-{x}^*(t)}{\epsilon}$.   Define $b(t)=-1$ if $t\in [t_0,t_1]$ and  $b(t)=0$ if $t\in (t_1,\infty)$.  
From \eqref{DEx_1} and \eqref{DEx_2}, we have
\begin{align}\label{DEy_1}
\dot{x}_{\epsilon}(t)-\dot{x}^*(t)&=f(x_{\epsilon}(t),t)-f(x^*(t),t)+ b(t) \epsilon.
\end{align}
Given $t$, by dividing both sides of \eqref{DEy_1} by $\epsilon$ and letting $\epsilon \to 0^+$ we have
\begin{align}\label{DEy_3}
\dot{y}(t)&=A(t)y(t)+b(t), \quad y(t_0)=0.
\end{align}

We now evaluate $U \equiv  \int_{t_0}^\infty e^{-\theta t} u(c^*(t), x^*(t)) dt$ and $U_{\epsilon}\equiv \int_{t_0}^\infty e^{-\theta t} u(c_{\epsilon}(t), x_{\epsilon}(t))  dt$. Applying the Mean Value Theorem, we have
$$\lim_{\epsilon \to 0^+} \frac{U_{\epsilon} - U}{\epsilon} =  \int_{t_0}^{\infty} e^{-\theta t} \left( \lim_{c \to 0^+} \frac{\partial u}{\partial c}(c, x^*(t)) \right) dt + \int_{t_0}^\infty e^{-\theta t} \frac{\partial u}{\partial x}(0, x^*(t)) y(t) \, dt.$$
%

By Assumption \ref{interior_assumption}.\ref{term2_bounded}, the second term in the right hand side is bounded while the first term is equal to $+\infty$ because of Inada condition. By consequence,  $\lim_{\epsilon \to 0^+} (U_{\epsilon} - U)/\epsilon=+\infty$. Hence, $U_{\epsilon} - U>0$ for all $\epsilon>0$ small enough. This strictly contradicts the premise that $c^*(t) = 0$ is an optimal control trajectory.  Lemma \ref{nozero_condition} has been proved.
\end{proof}
{\bf Step 2}: we prove that $c^*(t)>0$ almost
everywhere. From Step 1, we have $\lambda_0=1,\lambda(t)\not=0$. So, condition \eqref{optimal_c_general} becomes 
\begin{align*}
e^{-\theta t}u(c^*(t),x^*(t))+\lambda(t) \big(f(x^*(t),t)-c^*(t)\big)&\geq e^{-\theta t}u(c,x^*(t))+\lambda(t) \big(f(x^*(t),t)-c\big) \text{ } \forall c\geq 0.
\end{align*}
This implies that 
\begin{align*}
e^{-\theta t}u(c^*(t),x^*(t))-\lambda(t) c^*(t)&\geq e^{-\theta t}u(c,x^*(t))-\lambda(t)c \text{ } \forall c\geq 0.
\end{align*}
By consequence, the standard first-order condition gives  $e^{-\theta t} \frac{\partial u}{\partial c}(c^*(t), x^*(t)) - \lambda(t) \leq 0$  with complementary slackness: $c^*(t) \left[ e^{-\theta t} \frac{\partial u}{\partial c}(c^*(t), x^*(t)) - \lambda(t) \right] = 0$.

Suppose that there exists a Lebesgue measurable $E\subseteq \mathcal{T}$ with a strictly positive measure and $c^*(t)=0$ $\forall t\in E$. By combining $e^{-\theta t} \frac{\partial u}{\partial c}(c^*(t), x^*(t)) \leq \lambda(t)$ and Inada condition, we have $\lambda(t)=+\infty$ $\forall t\in E$. This is impossible because the function $\lambda$ is piecewise continuous.

To sum up, we have $c^*(t)>0$ almost
everywhere. Then,  condition \eqref{optimal_c_general} and $c^*(t)\in \rr_{++}$ imply that $\frac{\partial H}{\partial c}(x^*(t),c^*(t),t,\lambda(t))=0$. 
Therefore, we obtain \eqref{foc_general_1}. Note also that condition \eqref{Hh_derivative} becomes \eqref{foc_general_2}.


{\small
\bibliographystyle{apalike}
\bibliography{BDP_ref_bibtex}
}
\end{document}